\documentclass[11pt]{amsart}

\usepackage[T1]{fontenc}
\usepackage{amsmath,amssymb,amsthm}
\usepackage{hyperref}
\hypersetup{colorlinks=true,linkcolor=blue,citecolor=blue,urlcolor=blue}
\usepackage{orcidlink}

\theoremstyle{plain}
\newtheorem{theorem}{Theorem}
\theoremstyle{definition}

\newcommand{\Cs}{C_{\sigma}}

\begin{document}

\title[A counterexample to a conjecture of Cohen]{A counterexample to a
subadditivity conjecture of Cohen for Sophie Germain cyclic numbers}

\author{Josu\'{e} Alexander Ibarra\,\orcidlink{0009-0007-1274-6685}}
\email{alex@elninja.com}
\date{\today}

\begin{abstract}
An integer $n\ge 1$ is \emph{cyclic} if $\gcd(n,\varphi(n))=1$---equivalently, if
every group of order $n$ is cyclic---and \emph{Sophie Germain cyclic} if both $n$
and $2n+1$ are cyclic. Let $\Cs(N)$ count the Sophie Germain cyclic integers in
$[1,N]$. Cohen~\cite{Cohen2025} conjectured that $\Cs$ is subadditive,
$\Cs(m+n)\le\Cs(m)+\Cs(n)$ for all $1\le m\le n$ (his Conjecture~66), having
checked $m,n\le 10^{6}$ without finding a counterexample. We give one: at $m=31$,
$n=3928$,
\[
  \Cs(3959)=697 \;>\; 696 = \Cs(31)+\Cs(3928).
\]
The argument is short, and is verified by the Lean~4 kernel.
\end{abstract}

\maketitle

\section{The conjecture}

For an integer $n\ge 1$, let $\varphi(n)$ be Euler's totient. By a classical
theorem, every group of order $n$ is cyclic if and only if $\gcd(n,\varphi(n))=1$;
such $n$ are the \emph{cyclic numbers}
(OEIS~\href{https://oeis.org/A003277}{A003277}), beginning
$1,2,3,5,7,11,13,15,17,19,23,\dots$. By analogy with Sophie Germain primes,
Cohen~\cite{Cohen2025} calls $n$ \emph{Sophie Germain cyclic} when both $n$ and
$2n+1$ are cyclic, and writes
\[
  \Cs(N) = \#\{\, k : 1\le k\le N,\ k \text{ is Sophie Germain cyclic}\,\}
\]
for their counting function. The Sophie Germain cyclic numbers
(OEIS~\href{https://oeis.org/A397387}{A397387}) begin
$1,2,3,5,7,11,15,17,23,29,\dots$, so $\Cs(31)=10$.

Cohen studies cyclic-number analogs of classical statements about primes---among
them analogs of the prime number theorem, of bounds on prime gaps, and of the
Hardy and Littlewood conjectures---testing each against the cyclic numbers below
$10^{8}$. One such statement is the subadditivity of the prime counting function,
$\pi(m+n)\le\pi(m)+\pi(n)$, the second conjecture of Hardy and
Littlewood~\cite{HardyLittlewood1923}; it is now believed false, being
incompatible with the prime $k$-tuples conjecture~\cite{HensleyRichards1974}.
Cohen states the cyclic analog (his Conjecture~65) and disproves it, then states
the Sophie Germain version:

\begin{quote}
\textbf{Conjecture 66.} For all integers $1\le m\le n$,\quad
$\Cs(m+n)\le\Cs(m)+\Cs(n)$.
\end{quote}

He reports finding no counterexample for $m,n\le 10^{6}$ and leaves it open. It is
false.
\begin{theorem}\label{thm:main}
At $m=31$ and $n=3928$,
\[
  \Cs(3959)=697 \;>\; 696 = \Cs(31)+\Cs(3928),
\]
so Conjecture~66 fails. The witness $(31,3928)$ lies within the range
$m,n\le 10^{6}$.
\end{theorem}

\section{The counterexample}

Subadditivity $\Cs(m+n)\le\Cs(m)+\Cs(n)$ asks that no window of length $m$ contain
more Sophie Germain cyclic integers than the initial segment $[1,m]$. Here it
fails for a simple reason: the length-$31$ window $(3928,3959]$ contains
\emph{eleven} Sophie Germain cyclic integers, one more than the ten in $[1,31]$.

\begin{proof}[Proof of Theorem~\ref{thm:main}]
The Sophie Germain cyclic integers in $[1,31]$ are
\[
  1,\,2,\,3,\,5,\,7,\,11,\,15,\,17,\,23,\,29,
\]
ten in all, so $\Cs(31)=10$. The window $(3928,3959]$ contains eleven,
\[
  3929,\,3931,\,3935,\,3941,\,3943,\,3945,\,3947,\,3949,\,3953,\,3957,\,3959,
\]
so $\Cs(3959)=\Cs(3928)+11$. Therefore
\[
  \Cs(3959)=\Cs(3928)+11 \;>\; \Cs(3928)+10 = \Cs(31)+\Cs(3928). \qedhere
\]
\end{proof}

The value of $\Cs(3928)$ is never needed---only that the window to its right is
denser than $[1,31]$. As an illustration, $3929$ belongs to the window because it
is prime, giving $\varphi(3929)=3928$ and $\gcd(3929,3928)=1$, while
$2\cdot 3929+1=7859=29\cdot 271$ gives $\varphi(7859)=28\cdot 270=7560$ and
$\gcd(7859,7560)=1$; both $3929$ and $7859$ are cyclic, so $3929$ is Sophie Germain
cyclic. The other members are checked the same way. Counterexamples of this kind
are not rare; what makes this one notable is that it lies inside the range
$m,n\le 10^{6}$ Cohen reports searching.

\section{Verification}

The counterexample is elementary: each count is a finite enumeration the reader
can reproduce. We computed $\Cs(31)$, $\Cs(3928)$, and $\Cs(3959)$ with two
independent implementations of $\varphi$, and checked the definition against
Cohen's tabulated value $\Cs(598)=120$.

It is also verified formally. In the Lean~4 proof assistant over \texttt{mathlib},
a number is cyclic exactly when \texttt{Nat.Coprime n (Nat.totient n)}, Sophie
Germain cyclic when $2n+1$ is also, and $\Cs(N)$ is \texttt{Nat.count} of the
Sophie Germain cyclic integers below $N+1$. The negation of Conjecture~66 is
proved, and \texttt{\#print axioms} reports that it depends only on the three
standard axioms \texttt{propext}, \texttt{Classical.choice}, and
\texttt{Quot.sound}---in particular no \texttt{native\_decide}, so the result is
checked by the kernel itself, not by a compiled program. As in the proof above,
the formal argument splits $\Cs(3959)$ at $3928$ and evaluates only the eleven
window members, never the cumulative count $\Cs(3928)$. The source accompanies
this note as supplementary material and is self-contained over \texttt{mathlib}.

\section*{Remark}

Conjecture~66 is the Sophie Germain analog of Cohen's subadditivity conjecture for
ordinary cyclic numbers (his Conjecture~65), which he disproves; both are now
known to be false. This is what one would expect if the underlying conjecture
about primes---Hardy and Littlewood's $\pi(m+n)\le\pi(m)+\pi(n)$---is itself
false, as is generally believed~\cite{HensleyRichards1974}.

\section*{Acknowledgment}

The author thanks Joel E. Cohen for verifying the counterexample and confirming
that the search reported in~\cite{Cohen2025}, which found no counterexample to
Conjecture~66 for $m,n\le 10^{6}$, resulted from an error in his computer code
(personal communication, 2026).

\end{document}